\documentclass[12pt]{article}
\usepackage{amsmath,amsthm,amssymb,latexsym,color, dsfont}
\usepackage{hyperref}

\theoremstyle{plain}
\newtheorem{theor}{Theorem}
\newtheorem{prop}[theor]{Proposition}

\newtheorem{cor}[theor]{Corollary}
\newtheorem{lemma}[theor]{Lemma}
\theoremstyle{remark}

\def\R{{\mathbb R}}
\def\Prob{{\mathbb P}}
\def\N{{\mathbb N}}

\def\Vol{{\mathrm{Vol}}}
\def\Net{{\mathcal N}}
\def\Ill{{\mathcal I}}

\title{Randomized coverings of a convex body with its homothetic copies, and illumination}
\author{Galyna Livshyts\footnote{School of Mathematics, Georgia Institute of Technology, glivshyts6@math.gatech.edu}\;
and Konstantin Tikhomirov\footnote{Dept.\ of Math.\ and Stats., University of Alberta, ktikhomi@ualberta.ca. A part of this work was done when
K.T. visited GeorgiaTech in November, 2015.}}

\begin{document}

\maketitle

\begin{abstract}
We present a probabilistic model of illuminating a convex body by independently distributed light sources.
In addition to recovering C.A. Rogers' upper bounds for the illumination number, we improve previous estimates of
J.~Januszewski and M. Nasz\'odi for a generalized version of the illumination parameter.
\end{abstract}

\section{Introduction}

Given a convex body (i.e.\ a compact convex set with non-empty interior)
$K$ in $\R^n$ and points $p_1,p_2,\dots,p_m\in\R^n\setminus K$,
we say that the collection $\{p_1,p_2,\dots,p_m\}$ {\it illuminates $K$} if for any point $x$ on the boundary of $K$ there
is a point $p_i$ such that the line passing through $x$ and $p_i$ intersects the interior of $K$
at a point not between $p_i$ and $x$.
The {\it illumination number} $\Ill(K)$ is the cardinality of the smallest collection of points illuminating $K$.

The well known conjecture of H.~Hadwider \cite{H60},
independently formulated by I.~Gohberg and A.~Markus, asserts that $\Ill(K)\leq 2^n$ for any $n$-dimensional
convex body, with the equality attained for parallelotopes. The problem is known to be equivalent
to the question whether every convex body can be covered by at most $2^n$ smaller homothetic copies of itself
(see, for example, V.~Boltyanski, H.~Martini, P. S.~Soltan, \cite[Theorem~34.3]{BMS}).
For a detailed discussion of the problem and a survey of partial results, we refer to
\cite[Chapter~VI]{BMS}, K.~Bezdek \cite[Chapter~3]{B10} and a recent survey by K.~Bezdek and M. A.~Khan \cite{BK}.

An upper bound for the illumination number, which follows from a classical covering argument of C.A.~Rogers
\cite{R57}, is
\begin{equation}\label{eq: Ill number}
\Ill(K)\leq (n\log n+n\log\log n+5n)\frac{\Vol_n(K-K)}{\Vol_n(K)}
\end{equation}
(see, for example, K.~Bezdek \cite[Theorem~3.4.1]{B10}). Here, $\Vol_n(\cdot)$ is the Lebesgue measure in $\R^n$,
and $K-K$ is the Minkowski sum of $K$ and $-K$. Using the estimate of $\Vol_n(K-K)$ due to C.A.~Rogers
and G.C.~Shephard \cite{RS}, we get $\Ill(K)\leq (1+o(1)){2n\choose n} n\log n$. Moreover,
for a centrally-symmetric $K$ we clearly have $\Vol_n(K-K)=2^n\Vol_n(K)$, whence $\Ill(K)\leq (1+o(1))2^n n\log n$.

The proof of \eqref{eq: Ill number} based on C.A.~Rogers' covering of $\R^n$,
combines probabilistic and deterministic arguments,
and does not give much information about the arrangement of points illuminating $K$.
One of motivations for this work was to present a simple probabilistic model for the illumination,
which provides more data about the collection of the light sources.
In fact, we consider a more general question of covering a given convex body with its positive homothetic copies
of different sizes.
We prove the following:

\begin{prop}\label{pure randomization}
Let $n$ be a sufficiently large positive integer, $K$ be a
convex body in $\R^n$ with the origin in its interior and let numbers $(\lambda_i)_{i=1}^m$ satisfy
$\lambda_i\in (e^{-n},1)$ ($i=1,2,\dots,m$) and
$$\sum\limits_{i=1}^m{\lambda_i}^n\geq (n\log n+n\log\log n+4n)\frac{\Vol_n(K-K)}{\Vol_n(K)}.$$
For each $i$, let $X_i$ be a random vector uniformly distributed inside the set $K-\lambda_i K$,
so that $X_1,X_2,\dots,X_m$ are jointly independent.
Then the random collection of translates $\{X_i+\lambda_i K\}_{i=1}^m$ covers $K$ with probability at least $1-e^{-0.3n}$.
\end{prop}
As an easy corollary of the above statement, we obtain:
\begin{cor}\label{illumination}
Let $n$ be a large positive integer, and $K$ be a convex body in $\R^n$ with the origin in its interior.
Then there is a number $R>0$ depending only on $n$ with the following property:
Let $X$ be a random vector uniformly distributed over $K-K$, and let
$$
m:=\Bigl\lceil (n\log n+n\log\log n+5n)\frac{\Vol_n(K-K)}{\Vol_n(K)}\Bigr\rceil.
$$
Let $X_1,X_2,\dots,X_m$ be independent copies of $X$.
Then with probability at least $1-e^{-0.3 n}$ the collection $\{RX_1,RX_2,\dots,RX_m\}$
illuminates $K$.
\end{cor}

Let us note that illumination of convex sets by independent random light sources was previously considered
in literature. Namely, O.~Schramm \cite{S} used such a model to estimate the illumination number
for bodies of constant width; later, this approach was generalized by K.~Bezdek to
so-called fat spindle bodies \cite{B12}.

\bigskip

Proposition~\ref{pure randomization} allows us to study the following notion, closely related to the
illumination number.
For a convex body $K$ in $\R^n$, define $f_n(K)$
to be the least positive number such that for any sequence
$(\lambda_i)$ ($\lambda_i\in[0,1)$) with $\sum{\lambda_i}^n> f_n(K)$
there are points $x_i\in\R^n$ such that the collection of homothets $\{\lambda_i K+x_i\}$
covers $K$.
It was shown by A.~Meir and L.~Moser \cite{MM68} that $f_n\bigl([0,1]^n\bigr)=2^n-1$.
For an arbitrary convex body $K$, J.~Januszewski \cite{J03} showed that
$f_n(K)\leq (n+1)^n-1$.
Further, M. Nasz\'odi \cite{N10} showed that for any $K$ with its center of mass at the origin,
$$f_n(K)\leq 2^n\frac{\Vol_n(K+\frac{1}{2}K\cap(-K))}{\Vol_n(K\cap (-K))}\leq\begin{cases}
3^n,&\mbox{if $K=-K$},\\6^n,&\mbox{otherwise.}\end{cases}$$
We refer to P.~Brass, W.~Moser, J.~Pach \cite[p.~131]{BMP} for a more extensive discussion of this quantity.

Our Proposition~\ref{pure randomization}, together with a Rogers--type argument,
gives the following:
\begin{cor}\label{main result}
Let $n$ be a (large enough) positive integer and $K$ be a convex body in $\R^n$.
Then, with the quantity $f_n(K)$ defined above,
we have
$$f_n(K)\leq \Bigl\lceil(n\log n+n\log\log n+5n)\frac{\Vol_n(K-K)}{\Vol_n(K)}\Bigr\rceil.$$
\end{cor}
We remark, that together with the Rogers--Shephard bound
on the volume of the difference body from \cite{RS}, Corollary~\ref{main result} implies that
$$f_n(K)\leq \begin{cases}
2^nn\log n(1+o(1)),&\mbox{if $K=-K$},\\ \frac{1}{\sqrt{\pi n}}4^nn\log n(1+o(1)),&\mbox{otherwise.}\end{cases}$$

We note that several other illumination-related quantities, different from $f_n(K)$, were considered
in literature. We refer, in particular, to \cite{Bezdek92, BL, Sw}.

\bigskip

Let us emphasize that proofs of all the above statements are very simple.
The purpose of this note is to put forward a randomized model for studying the illumination number
and its generalizations.
We believe that such viewpoint to the Illumination Problem will prove useful.
We give a proof of Proposition~\ref{pure randomization} in Section~\ref{s: main res}, whereas the corollaries
are derived in Section~\ref{s: cor}.

\section{Notation and preliminaries}

The standard vector basis in $\R^n$ is denoted by $\{e_1,e_2,\dots,e_n\}$.
For a non-zero vector $v\in\R^n$, $v^{\perp}$ is the hyperplane orthogonal to $v.$
By $B_\infty^n$ we denote the cube $[-1,1]^n$.

Given two sets $A,B\subset\R^n$, the Minkowski sum $A+B$ is defined as
$$A+B:=\{x+y\,:\, x\in A, \, y\in B\}.$$
Let $K$ be a convex body, and let $\varepsilon>0$.
Then an $\varepsilon$-net $\Net$ on $K$ is a set of points
$\{x_i\}\subset \R^n$ such that the collection of convex sets $\{x_i+\varepsilon K\}$ covers $K$.

We make the following observation.

\begin{lemma}\label{K-K}
Let $n\geq 2$ be an integer. For every convex body $K$ in $\R^n$ and for every $\lambda\in [0,1]$ we have
$$\Vol_n(K-\lambda K)\leq(1+\lambda)^n \frac{\Vol_n(K-K)}{2^n}.$$
\end{lemma}
\begin{proof}
Observe that
$$\Vol_n(K-\lambda K)=(1+\lambda)^n \Vol_n\left(\mu K-(1-\mu)K\right),$$
where $\mu=\frac{1}{1+\lambda}\in [0,1].$ The lemma follows from the fact that
\begin{equation}\label{maximizer}
\max_{\nu\in [0,1]}\Vol_n\left(\nu K-(1-\nu)K\right)=\Vol_n\left(\frac{K}{2}-\frac{K}{2}\right)=\frac{\Vol_n(K-K)}{2^n}.
\end{equation}
To establish \eqref{maximizer}, let us consider an auxiliary $(n+1)$-dimensional convex set ${\mathcal C}$ given by
$${\mathcal C}:=\textrm{conv}\left(K\times\{0\}\cup (-K)\times\{1\}\right)$$
(let us remark that the use of such auxiliary sets is rather standard and goes back at least to
C.A.~Rogers and G.C.~Shephard \cite{RS58}; also, see S.~Artstein-Avidan \cite{ShArt}).
Observe that for any $\nu\in[0,1]$ we have
$${\mathcal C}\cap (e_{n+1}^{\perp}+\nu e_{n+1})=\left(\nu K-(1-\nu)K\right)\times\{\nu\}.$$
The set ${\mathcal C}$ is convex and symmetric with respect to $\frac{1}{2}e_{n+1}$. Hence, the $n$-dimensional
section of ${\mathcal C}$ given by the hyperplane $e_{n+1}^{\perp}+\frac{1}{2} e_{n+1}$, has maximal
$n$-dimensional volume among all other sections of ${\mathcal C}$ parallel to it.
\end{proof}

Next, for the reader's convenience we provide a standard estimate of the covering number.
\begin{lemma}\label{net card}
Let $n$ be a sufficiently large positive integer,
and let $K$ be a convex body in $\R^n$ with the origin in its interior. Then for any $\varepsilon\in(0,1]$
there exists an $\varepsilon$-net on $K$ of cardinality at most $\left(\frac{5}{\varepsilon}\right)^n$.
\end{lemma}
\begin{proof}
By the Rogers--Zong lemma \cite{RZ}, there exists an $\varepsilon$-net of cardinality at most
$$\frac{\Vol_n(K-\varepsilon K)}{\Vol_n(\varepsilon K)}\bigl(n\log n+n\log\log n+5n\bigr).$$
By Lemma \ref{K-K}, along with the Rogers--Shephard lemma \cite{RS}, we estimate 
$$\frac{\Vol_n(K-\varepsilon K)}{\Vol_n(\varepsilon K)}\bigl(n\log n+n\log\log n+5n\bigr)\leq (1+o(1))
\frac{4^n(1+\varepsilon)^n n\log n}{\sqrt{\pi n} 2^n \varepsilon^n}\leq \left(\frac{5}{\varepsilon}\right)^n.$$
\end{proof}

\section{Proof of Proposition~\ref{pure randomization}}\label{s: main res}

Let $(\lambda_i)_{i=1}^m$ satisfy the assumptions of the proposition,
and let $X_1,X_2,\dots,X_m$ be jointly independent random vectors, where each $X_i$ is uniformly distributed in
$K-\lambda_i K$.
Assume without loss of generality that $\Vol_n(K)=1$.

We shall estimate the probability
\begin{equation*}
\Prob\Bigl\{K\subset \bigcup_{i=1}^m (X_i+\lambda_i K)\Bigr\}.
\end{equation*}
Let $\varepsilon\in(0,1]$ be chosen later, and
consider an $\varepsilon$-net $\Net$ on $K$ of cardinality at most $\left(\frac{5}{\varepsilon}\right)^n$
(which exists according to Lemma~\ref{net card}).
We can safely assume that $\Net\subset K-\varepsilon K$.
Observe that
\begin{equation*}
\Prob\Bigl\{\exists x\in K\,:\,
x\notin\bigcup_{i=1}^m (X_i+\lambda_i K)\Bigr\}
\leq \Prob\Bigl\{\exists y\in \Net\,:\, y\notin \bigcup_{i=1}^m (X_i+(\lambda_i-\varepsilon)_+ K)\Bigr\},
\end{equation*}
where $(\lambda_i-\varepsilon)_+:=\max(0,\lambda_i-\varepsilon)$. Hence, by the union bound,
\begin{equation*}
\Prob\Bigl\{K\subset \bigcup_{i=1}^m (X_i+\lambda_i K)\Bigr\}
\geq 1-\left(\frac{5}{\varepsilon}\right)^n \max_{y\in \Net}
\Prob\Bigl\{y\not\in \bigcup_{i=1}^m (X_i+(\lambda_i-\varepsilon)_+ K)\Bigr\}.
\end{equation*}
Fix any $y\in\Net$ and note that
\begin{equation*}
\Prob\Bigl\{y\not\in \bigcup_{i=1}^m (X_i+(\lambda_i-\varepsilon)_+ K)\Bigr\}
=\prod_{i=1}^m \left(1-\Prob\bigl\{y\in X_i+(\lambda_i-\varepsilon)_+K\bigr\}\right).
\end{equation*}
Further, observe that 
\begin{equation*}
\Prob\bigl\{y\in X_i+(\lambda_i-\varepsilon)_+K\bigr\}
=\Prob\bigl\{X_i\in y-(\lambda_i-\varepsilon)_+K\bigr\}=\frac{{(\lambda_i-\varepsilon)_+}^n}{\Vol_n(K-\lambda_i K)},
\end{equation*}
where the last equality is due to the fact that $\Vol_n(K)=1$ and that
$$y-(\lambda_i-\varepsilon)K\subset K-\lambda_i K$$ 
whenever $\lambda_i> \varepsilon$.
Combining the above relations, we obtain
\begin{equation*}
\Prob\Bigl\{K\subset \bigcup_{i=1}^m (X_i+\lambda_i K)\Bigr\}
\geq 1-\left(\frac{5}{\varepsilon}\right)^n \prod_{i=1}^m \left(1-\frac{{(\lambda_i-\varepsilon)_+}^n}{\Vol_n(K-\lambda_i K)}\right).
\end{equation*}

Now, our aim is to show that under the assumptions of the proposition there exists an $\varepsilon$ such that
\begin{equation*}
\left(\frac{5}{\varepsilon}\right)^n \prod_{i=1}^m \left(1-\frac{{(\lambda_i-\varepsilon)_+}^n}
{\Vol_n(K-\lambda_i K)}\right)\leq e^{- 0.3n},
\end{equation*}
or, equivalently,
$$\sum_{i=1}^m \log\left(1-\frac{{(\lambda_i-\varepsilon)_+}^n}
{\Vol_n(K-\lambda_i K)}\right)\leq n\log\Bigl(\frac{\varepsilon}{5}\Bigr)-0.3 n.$$
In view of Lemma~\ref{K-K}, and the relation $\log(1-t)\leq -t$ valid for all $t\geq 0$,
we have
\begin{align*}
\sum_{i=1}^m \log\left(1-\frac{{(\lambda_i-\varepsilon)_+}^n}
{\Vol_n(K-\lambda_i K)}\right)
&\leq
\sum_{i=1}^m \log\left(1-\frac{2^n{(\lambda_i-\varepsilon)_+}^n}
{(1+\lambda_i)^n\Vol_n(K-K)}\right)\\
&\leq
-\sum_{i=1}^m\frac{2^n{(\lambda_i-\varepsilon)_+}^n}
{(1+\lambda_i)^n\Vol_n(K-K)}.
\end{align*}
Thus, in order to prove the proposition, it is sufficient to show that for some $\varepsilon\in(0,1]$ we have
\begin{equation}\label{eq: enough}
n\log 5+n\log\frac{1}{\varepsilon}-\sum_{i=1}^m \frac{2^n {(\lambda_i-\varepsilon)_+}^n}{(1+\lambda_i)^n \Vol_n(K-K)}\leq -0.3 n.
\end{equation}

Let $A_n:=1-\frac{4\log n}{n}$.
We consider two complimentary subsets of $\{1,2,\dots,m\}$:
\begin{align*}
L_1&=\bigl\{i\leq m:\,\lambda_i\geq A_n\bigr\}\\
L_2&=\bigl\{i\leq m:\,\lambda_i< A_n\bigr\}.
\end{align*}
The rest of the proof splits into two cases.

\textbf{Case 1:}
\begin{equation}\label{eq: sum 1}
\sum_{i\in L_1} {\lambda_i}^n\geq \left(1-\frac{1}{\log n}\right)\sum_{i=1}^m {\lambda_i}^n.
\end{equation}
Choose $\varepsilon:=\frac{A_n}{n\log n}$.
Then 
$$(\lambda_i-\varepsilon)^n\geq {\lambda_i}^n\left(1-\frac{1}{\log n}\right)\;\;\mbox{ for all }i\in L_1,$$
whence, using the condition $\lambda_i\leq 1$ together with \eqref{eq: sum 1} and the condition
on the sum of ${\lambda_i}^n$, we get
\begin{align*}
\sum_{i\in L_1} \frac{2^n (\lambda_i-\varepsilon)^n}{(1+\lambda_i)^n \Vol_n(K-K)}
&\geq\Bigl(1-\frac{1}{\log n}\Bigr)\sum_{i\in L_1} \frac{{\lambda_i}^n}{\Vol_n(K-K)}\\
&\geq \Bigl(1-\frac{1}{\log n}\Bigr)^2\sum_{i=1}^m \frac{{\lambda_i}^n}{\Vol_n(K-K)}\\
&\geq \Bigl(1-\frac{1}{\log n}\Bigr)^2(n\log n+n\log\log n+4n).
\end{align*}
It is easy to check, using the above inequality, that \eqref{eq: enough} is satisfied, and Case 1 is settled.

\textbf{Case 2:}
$$\sum_{i\in L_2} {\lambda_i}^n> \frac{1}{\log n}\sum_{i=1}^m {\lambda_i}^n.$$
Set $\varepsilon:=\frac{e^{-n}}{n\log n}$.
By the assumption of the proposition, $\lambda_i\geq e^{-n}$ for all $i$. Hence, 
$${(\lambda_i-\varepsilon)_+}^n\geq \left(1-\frac{1}{\log n}\right){\lambda_i}^n,\;\;i\leq m,$$ 
and the left hand side of \eqref{eq: enough} is less than
\begin{align*}
&n\log 5+n^2+n\log n+n\log\log n-\left(1-\frac{1}{\log n}\right)\frac{2^n}{\Vol_n(K-K)}\sum_{\lambda_i\in L_2}\frac{ {\lambda_i}^n}{(1+\lambda_i)^n}\\
&n\log 5+n^2+n\log n+n\log\log n-\frac{1}{\log n}\left(1-\frac{1}{\log n}\right)\frac{2^n}{\Vol_n(K-K)}\sum_{i=1}^m\frac{ {\lambda_i}^n}{(1+\lambda_i)^n}\\
&\leq
n\log 5+n^2+n\log n+n\log\log n-\left(1-\frac{1}{\log n}\right)\frac{2^n}{(1+A_n)^n}\frac{n\log n+n\log\log n+4n}{\log n}\\
&\ll -0.3 n.
\end{align*}
Here, we used the definition of $A_n$,
and the assumption on the sum of ${\lambda_i}^n$.
Thus, Case 2 is settled, and the proof of Proposition~\ref{pure randomization} is complete.

\section{Proof of the Corollaries}\label{s: cor}

\subsection{Proof of Corollary~\ref{illumination}}

First, we recall that if $K$ is a convex body
and $K\subset \cup_{i=1}^m {\rm int}(K)+x_i$
for some non-zero vectors $x_1,x_2,\dots,x_m$,
then there exists $R>0$ such that points $Rx_1,Rx_2,\dots,Rx_m$ illuminate $K$
(see, for example, \cite[proof of Theorem~34.3]{BMS}).
It is not difficult to verify the following quantitative version of above observation:
If $K\subset \cup_{i=1}^m (1-\varepsilon)K+x_i$ for some $\varepsilon\in (0,1)$,
then $Rx_1,Rx_2,\dots,Rx_m$ illuminate $K$ whenever $R> \frac{1}{\varepsilon}$.

Let $n$ be a sufficiently large integer, denote
$$
m:=\Bigl\lceil (n\log n+n\log\log n+5n)\frac{\Vol_n(K-K)}{\Vol_n(K)}\Bigr\rceil,
$$
and select $\varepsilon=\varepsilon(n)>0$ small enough so that
$$
m(1-\varepsilon)^n\geq (n\log n+n\log\log n+4n)\frac{\Vol_n(K-K)}{\Vol_n(K)}.
$$
Let $X_1,X_2,\dots,X_m$ be i.i.d.\ uniformly distributed in $K-K$. Then, by
Proposition~\ref{pure randomization},
with probability at least $1-e^{-0.3 n}$ the collection $\{X_i+(1-\varepsilon)K\}_{i=1}^m$
forms a covering of $K$; hence, for any fixed $R>\frac{1}{\varepsilon}$,
the vectors $RX_1,RX_2,\dots,R X_m$ illuminate $K$ with probability at least $1-e^{-0.3n}$.

\subsection{Proof of Corollary~\ref{main result}}

The next lemma is a variation of the well known theorem of C.A.~Rogers \cite{R57}
on economical coverings of $\R^n$ with convex bodies. 
\begin{lemma}\label{Rogers variation}
Let $n$ be a sufficiently large positive integer and
$K$ be a convex body in $\R^n$ such that $-\frac{1}{n}K\subset K\subset B_\infty^n$.
Further, let $L\geq n^2$ and let $\left(\lambda_i\right)_{i=1}^M$ be a sequence of numbers in $[1/2,1]$ with
$$\sum_{i=1}^M \Vol_n(\lambda_i K)\geq (n\log n+n\log\log n+5n)\Vol_n(L B_\infty^n).$$
Then there exists a translative covering of $L B_\infty^n$ by $\{\lambda_i K\}_{i=1}^M$.
\end{lemma}
\begin{proof}
The proof to a large extent follows \cite{R57}; we provide it only for reader's convenience.
Let $M'\leq M$ be the least number such that
$$\sum_{i=1}^{M'} \Vol_n(\lambda_iK)\geq (n\log n+n\log\log n+4n)\Vol_n(L B_\infty^n),$$
and let $Y_i$ ($1\leq i\leq M'$) be independent random vectors uniformly distributed in $L B_\infty^n-2K$.
For any point $x\in L B_\infty^n-K$, using the condition $K\subset B_\infty^n$, we have
\begin{align*}
\Prob& \Bigl\{x\notin \Bigl(\lambda_i-\frac{1}{2n\log n}\Bigr) K+Y_i\;\;\mbox{for all}\;\;i\leq M'\Bigr\}\\
&=\prod_{i=1}^{M'}\Bigl(1-\Prob\Bigl\{Y_i\in x -\Bigl(\lambda_i-\frac{1}{2n\log n}\Bigr)K\Bigr\}\Bigr)\\
&\leq\prod_{i=1}^{M'}\Biggl(1-\Bigl(1-\frac{1}{n\log n}\Bigr)^n\frac{\Vol_n\bigl(\lambda_i K\bigr)}{\Vol_n(L B_\infty^n-2K)}\Biggr),
\end{align*}
where the last inequality is due to the fact that $x -\bigl(\lambda_i-\frac{1}{2n\log n}\bigr)K\subset L B_\infty^n-2K$.
Further, we have
\begin{align*}
\prod_{i=1}^{M'}&\Biggl(1-\Bigl(1-\frac{1}{n\log n}\Bigr)^n\frac{\Vol_n\bigl(\lambda_i K\bigr)}{\Vol_n(L B_\infty^n-2K)}\Biggr)\\
&\leq\exp\Biggl(-\Bigl(1-\frac{1}{n\log n}\Bigr)^n
\sum\limits_{i=1}^{M'}\frac{\Vol_n\bigl(\lambda_i K\bigr)}{\Vol_n(L B_\infty^n-2K)}\Biggr)\\
&\leq\exp\Biggl(-\frac{L^n}{(L+2)^n}\Bigl(1-\frac{1}{n\log n}\Bigr)^n
\sum\limits_{i=1}^{M'}\frac{\Vol_n\bigl(\lambda_i K\bigr)}{\Vol_n(L B_\infty^n)}\Biggr)\\
&\leq\exp\Biggl(-\frac{1}{(1+2n^{-2})^n}\Bigl(1-\frac{1}{\log n}\Bigr)
\bigl(n\log n+n\log\log n+4n\bigr)\Biggr)\\
&\leq \exp(-n\log n-n\log\log n-2n),
\end{align*}
where in the last inequality we used the assumption that $n$ is large.
Hence, there is a non-random collection of vectors $\{y_i\}_{i=1}^{M'}$ such that
\begin{align*}
\Vol_n&\Bigl((L B_\infty^n-K)\setminus \bigcup_{i=1}^{M'}
\Bigl(\Bigl(\lambda_i-\frac{1}{2n\log n}\Bigr)K+y_i\Bigr)\Bigr)\\
&\leq \exp(-n\log n-n\log\log n-2n)\Vol_n(LB_\infty^n-K).
\end{align*}
Suppose that $S:=L B_\infty^n\setminus \bigcup_{i=1}^{M'}\bigl(\lambda_i K+y_i\bigr)$
is non-empty. Let $\Net$ be a maximal discrete subset of $S$ such that
$$\Bigl(y-\frac{1}{2n\log n}K\Bigr)\cap \Bigl(y'-\frac{1}{2n\log n}K\Bigr)=\emptyset\;\;\mbox{for all}\;\;y\neq y'\in\Net.$$
Note that $\Net-\frac{1}{2n\log n}K\subset (L B_\infty^n-K)\setminus \bigcup_{i=1}^{M'}
\bigl((\lambda_i-\frac{1}{2n\log n})K+y_i\bigr)$, whence
\begin{equation}\label{eq: net card}
|\Net|\leq \frac{\exp(-n\log n-n\log\log n-2n)\Vol_n(LB_\infty^n-K)}{(2n\log n)^{-n}\Vol_n(K)}
\leq \frac{\Vol_n(LB_\infty^n-K)}{e^{n}\Vol_n(K)}.
\end{equation}
On the other hand, by the choice of $\Net$ and in view of the inclusion $-\frac{1}{n}K\subset K$, we have
$$S\subset \Net+\frac{1}{2n\log n}(K-K)\subset \Net+\frac{1}{2}K.$$
Finally, from the choice of $M'$ it follows that
$$\sum_{i=M'+1}^{M} \Vol_n(\lambda_i K)\geq \Vol_n(L B_\infty^n)-\Vol_n(K),$$
whence, using \eqref{eq: net card}, $M-M'\geq \frac{\Vol_n(L B_\infty^n)-\Vol_n(K)}{\Vol_n(K)}\geq |\Net|$.
It remains to define the points $\{y_i\}_{i=M'+1}^M$ so that $\{y_i\}_{i=M'+1}^M=\Net$;
then the collection $\{y_i+\lambda_i K\}_{i=1}^M$ covers $LB_\infty^n$.
\end{proof}

\begin{proof}[Proof of Corollary~\ref{main result}]
Let $(\lambda_i)_{i=1}^\infty$ be a sequence of numbers in $(0,1)$ such that
$$\sum\limits_{i=1}^\infty\Vol_n(\lambda_i K)> (n\log n+n\log\log n+5n)\Vol_n(K-K).$$
We need to show that in this case there exists a covering of $K$ of the form $\{y_i+\lambda_i K\}_{i=1}^\infty$.
We will assume that $\sum\limits_{i=1}^\infty\Vol_n(\lambda_i K)<\infty$.
First, suppose that
$$\sum\limits_{i:\lambda_i\geq n^{-5}}\Vol_n(\lambda_i K)\geq (n\log n+n\log\log n+4n)\Vol_n(K-K).$$
Then the result immediately follows from Proposition~\ref{pure randomization}.

Otherwise,
\begin{equation}\label{torefer}
\sum\limits_{i:\lambda_i< n^{-5}}\Vol_n(\lambda_i K)\geq \Vol_n(K-K)\geq 2^n\Vol_n(K).
\end{equation}
Further, without loss of generality (for example, by applying John's theorem \cite{J48, B92}
together with an appropriate affine transformation) we can assume that
$B_\infty^n\subset K\subset n^{3/2}B_\infty^n$.
Let $\{x_j+n^{-3/2}B_\infty^n\}_{j=1}^N$ be a minimal covering of $K$ by cubes with pairwise disjoint interiors.
Let us define subsets $I_k\subset\N$ by
$$I_k:=\bigl\{i\in\N:\,\lambda_i n^5\in (2^{-k},2^{-k+1}]\bigr\},\;\;k=1,2,\dots$$
and for every $k\in\N$ set
$$F(k):=\Bigl\lfloor\frac{\sum\nolimits_{i\in I_k}\Vol_n(\lambda_i K)}{(n\log n+n\log\log n+6n)
\Vol_n(2^{-k+1}n^{-3/2}B_\infty^n)}\Bigr\rfloor.$$
Further, for those $k$ with $F(k)>0$, we let $I_k^\ell$ ($\ell=1,2,\dots,F(k)$)
be a partition of $I_k$ such that
\begin{equation}\label{kell partition}
\sum\limits_{i\in I_k^\ell}\Vol_n(\lambda_i K)\geq (n\log n+n\log\log n+5n)\Vol_n(2^{-k+1}n^{-3/2}B_\infty^n),\;\;\ell=1,2,\dots,F(k).
\end{equation}
Note that such partitions can always be constructed as the volume of each $\lambda_i K$ ($i\in I_k$)
is negligible compared to $\Vol_n(2^{-k+1}n^{-3/2}B_\infty^n)$.
Further,
$$\sum\limits_{k:F(k)<n}\sum\limits_{i\in I_k}\Vol_n(\lambda_i K)
\leq 4n^2\log n\,\Vol_n(n^{-3/2}B_\infty^n)\ll n^{-n}\Vol_n(K),$$
whence, by (\ref{torefer})
\begin{align*}
\sum\limits_{k:F(k)\geq n}F(k)\Vol_n(2^{-k+1}n^{-3/2}B_\infty^n)
&\geq \sum\limits_{k:F(k)\geq n}
\frac{n}{n+1}\frac{\sum\nolimits_{i\in I_k}\Vol_n(\lambda_i K)}{(n\log n+n\log\log n+6n)}\\
&\geq \frac{n}{n+1}\frac{2^n\Vol_n(K)}{n\log n+n\log\log n+6n}-n^{-n}\Vol_n(K)\\
&> \Vol_n\Bigl(\bigcup\limits_{j=1}^N (x_j+n^{-3/2}B_\infty^n)\Bigr),
\end{align*}
where the last inequality follows from the trivial observation
$$\bigcup\limits_{j=1}^N (x_j+n^{-3/2}B_\infty^n)\subset (1+2n^{-3/2})K$$
and our assumption that $n$ is large.

The last relation implies that there exists a finite collection of translates
$$\mathcal C:=\bigl\{y_k^\ell+2^{-k+1}n^{-3/2}B_\infty^n,\;k:F(k)\geq n,\;\ell=1,2,\dots,F(k)\bigr\}\;\;(y_k^\ell\in\R^n)$$
such that the union of the cubes from $\mathcal C$ covers $\bigcup\limits_{j=1}^N (x_j+n^{-3/2}B_\infty^n)$,
hence, $K$.
For each cube from $\mathcal C$ we construct a translative covering by sets
$\{\lambda_i K\}_{i\in I_k^\ell}$ using condition \eqref{kell partition} and Lemma~\ref{Rogers variation}.
This completes the proof.
\end{proof}

\end{document}